\pdfoutput=1

\documentclass{elsarticle}

\usepackage{enumerate}
\usepackage{lmodern, microtype}

\usepackage{graphicx}

\usepackage{amsmath, amsthm, amssymb, bm, mathtools, xfrac}

\usepackage{verbatim}

\usepackage[english]{babel}

\usepackage[%
pdfpagelabels=true,
unicode=true,
colorlinks=false,
pdfborder={0 0 0}
]{hyperref}

\hypersetup{
	pdfauthor={Julius Je\ss{}berger},
	pdftitle={Existence of weak solutions for inhomogeneous generalized Navier-Stokes equations},
	pdflang={en} 
}

\numberwithin{equation}{section} 

\newtheorem{thm}[equation]{Theorem}
\newtheorem{lem}[equation]{Lemma}
\newtheorem{cor}[equation]{Corollary}

\newtheorem{assum}[equation]{Assumption}
\newtheorem{defi}[equation]{Definition}
\newtheorem{rem}[equation]{Remark}

\newcommand{\abs}[1]{\left\lvert#1\right\rvert}	  
\newcommand{\Bog}{\mathcal{B}}			
\renewcommand{\d}{^{\ast}}				
\newcommand{\del}{\partial}				
\renewcommand{\div}{\operatorname{div}}	
\newcommand{\injto}{\xhookrightarrow{}} 
\newcommand{\fp}[2]{\langle #1, #2 \rangle}       
\newcommand{\N}{\mathbb{N}}				
\newcommand{\norm}[1]{\left\lVert#1\right\rVert}  
\newcommand{\bignorm}[1]{\big\lVert#1\big\rVert}  
\renewcommand{\phi}{\varphi}			
\newcommand{\R}{\mathbb{R}}				
\newcommand{\Rd}{\mathbb{R}^d}			
\renewcommand{\S}{\boldsymbol{\mathcal{S}}}		  
\newcommand{\A}{\boldsymbol{\mathcal{A}}}		  
\renewcommand{\vec}[1]{\boldsymbol{#1}}	
\newcommand{\weakstarto}{\overset{*}{\rightharpoonup}}		
\newcommand{\weakto}{\rightharpoonup}						

\begin{document}

\begin{center}
	{ \Large \bfseries 
		Existence of weak solutions for inhomogeneous generalized Navier-Stokes equations
	}\\[0.4cm]
	{Julius~Je\ss{}berger and Michael~R\r{u}\v{z}i\v{c}ka}
	
	{\today}\\[0.4cm]
	
	{Institute of Applied Mathematics, Albert-Ludwigs-University Freiburg, Ernst-Zermelo-Str. 1, D-79104 Freiburg, Germany}
	
	{E-mail addresses: julius.jessberger@gmail.com, rose@mathematik.uni-freiburg.de}
\end{center}

\noindent\textbf{Abstract}\hspace{0.15cm}
We prove existence of weak solutions for the fully inhomogeneous, stationary generalized Navier-Stokes equations for shear-thinning fluids. Our proof is based on the theory of pseudomonotone operators and the Lipschitz truncation method, whose application is presented as a general result. Our approach requires a smallness and a regularity assumption on the data; we show that this is inevitable in the framework of pseudomonotone operators.

\noindent\textbf{Keywords}\hspace{0.15cm}
Generalized Newtonian fluid, pseudomonotone operator, existence of weak solutions, inhomogeneous problem.

\noindent\textbf{AMS Classifications (2020)}\hspace{0.15cm}
35Q35, 35B45, 35J92, 76D03.

\section{Introduction} \label{sec:intro}

Motivated by the equations describing the steady motion of
generalized Newtonian fluids we study the following fully
inhomogeneous system 
\begin{equation} \label{eq:main_problem} \begin{split}
- \div \S(\vec{Dv}) + \vec{v} \cdot \nabla \vec{v} + \nabla \pi &= \vec{f}, \\
\div \vec{v} &= g_1, \\
\vec{v}_{\vert \del \Omega} &= \vec{g_2}.
\end{split} \end{equation} In this setting, $\S$ is an extra stress
tensor with $p$-$\delta$-structure, $\vec{v}$ is the velocity field
with its symmetric gradient $\vec{Dv}$, $\pi$ is the pressure,
$\vec{f}$ is the external force and $g_1$ and $\vec{g_2}$ are data
on a sufficiently regular bounded domain
$\Omega \subset \Rd$ of dimension $d \in \{2, 3\}$.

Since \eqref{eq:main_problem}$_1$ leads to a pseudomonotone and
coercive operator in the homogeneous case $g_1 = 0$,
$\vec{g}_2=\vec{0}$ and $p >\frac {3d}{d+2}$ (cf.~\cite{lions-quel})
and in the shear-thickening case $p>2$ (cf.~\cite{r-mol-inhomo}), the
existence of weak solutions $(\vec{v}, \pi)$ to
\eqref{eq:main_problem} follows directly from the theory of
pseudomonotone operators in these cases. This approach can be adapted to the situation of homogeneous data and very low values of $p$: if $g_1 = 0$,
$\vec{g}_2=\vec{0}$ and $p \in (\frac{2d}{d+2}, \frac{3d}{d+2}]$, one can construct approximate solutions by the theory of pseudomonotone operators and prove their convergence with the Lipschitz truncation method (cf. \cite{dms}, \cite{fms2}). In the case $p=2$, we have to
deal with the fully inhomogeneous steady Navier-Stokes equations which
are studied intensively (cf.~\cite{Galdi}) and where the existence of
solutions is known under appropriate smallness conditions.  In the
shear-thinning, inhomogeneous case, i.e.~if $p \in (1, 2)$ and the
data $g_1$, $\vec{g_2}$ do not vanish, the coercivity of the elliptic
term is weaker than the growth of the convective term, i.e.~we are in
the supercritical
case. 
This situation is treated in \cite{mikelic}, \cite{Sin} for
$g_1=0$. In \cite{Sin} even the case of electrorheological fluids is
covered. The result there is based on a nice smallness argument (\cite[Lemma
3.2]{Sin}), which is applied to estimate the convective term. 
Since we did not understand the application of this lemma in detail,
we give a different proof of local coercivity here. Our main result
shows the existence of weak solutions of the fully inhomogeneous
problem \eqref{eq:main_problem} in the shear-thinning case under appropriate smallness conditions
involving higher regularity of the data.

The paper is organised as follows: by representing the inhomogeneous
data by a fixed function $\vec{g}$ (Subsection \ref{sub:div_eq}),
\eqref{eq:main_problem} turns into a homogeneous problem. We
investigate the newly formed elliptic and convective terms in
Subsections \ref{sub:extra_stress} and \ref{sub:conv_term}. Then we
conclude properties and local coercivity of the whole system and prove
existence of solutions (Subsection \ref{sub:existence}).  In the case
$p \in (\frac{2d}{d+2}, \tfrac{3d}{d+2})$ we use the Lipschitz
truncation method in order to establish convergence of approximate
solutions. This step is presented as an abstract statement, Theorem~\ref{thm:ident_limits}, which should fit to more general situations.
In contrast to \cite{Sin}, we had to require additional regularity of
the data in our proof of local coercivity. We discuss this issue in
Subsection \ref{sub:less_regularity} and prove that the additional
regularity assumption is necessary in the framework of pseudomonotone
operator theory.

The results presented here are based on the thesis \cite{MA} of the
first author.

\section{Preliminaries}

\subsection{Notation} \label{sub:notation} We work on a bounded
Lipschitz domain $\Omega \subset \Rd$, $d \in \{2, 3\}$, with
possesses an exterior normal $\vec{\nu}$.  Points and scalar-valued
quantities are written in normal letters whereas vector- and matrix-valued functions, variables and operators are denoted in bold
letters. The space of symmetric quadratic matrices is denoted as
$\R_\textit{sym}^{d\times d}$.

We use standard Lebesgue measure and integration theory. For a ball $B$, we
denote the ball with the same center and the double radius
by~$2B$. The characteristic function of a set $S \subset \Rd$ is
called $\chi_S$. 

We use standard notation for Lebesgue and Sobolev spaces. Due to
\cite{Gagliardo}, there exists a well-defined, surjective trace
operator $W^{1,p}(\Omega) \to W^{1-\frac{1}{p}, p}(\del \Omega)$ that
assigns boundary values to a Sobolev function.  We denote by
$L_0^p(\Omega)$ the subspace of $L^p(\Omega)$ of functions with mean
value zero and by $V_p$ the subspace of $ W_0^{1,p}(\Omega)$ of vector
fields with zero boundary values and zero divergence.  For a
vector-valued function $\vec{v} \in W^{1,p}(\Omega)$, the definition
of the (weak) gradient field follows the convention
$(\nabla \vec{v})_{ij} = \del_j v_i$ and the symmetric gradient is
defined as
$\vec{Dv} := \frac{1}{2} (\nabla\vec{v}+\nabla\vec{v}^\top)$.  On
$W_0^{1,p}(\Omega)$ and on $V_p$, we may work with the symmetric
gradient norm $\norm{\vec{D} \cdot}_p$, thanks to Poincar\'e's and
Korn's inequalities.

The dual of some Banach space $X$ is denoted as $X\d\!$ and
$\fp{\cdot}{\cdot}$ denotes their canonical dual pairing.  For an
exponent $p \in [ 1, \infty ]$, we define its conjugate exponent
$p' \in [ 1, \infty ]$ via $\frac{1}{p} + \frac{1}{p'} = 1$ and use
the duality $L^{p'}(\Omega) = L^p(\Omega)\d$ for $p <
\infty$. Finally, we define the critical Sobolev exponent
$p\d := \frac{pd}{d-p} \in (p, \infty)$ for~$p < d$.

\subsection{The divergence equation} \label{sub:div_eq} In order to
fulfil the boundary and divergence conditions in
\eqref{eq:main_problem}, we follow the usual ansatz
$\vec{v} = \vec{u} + \vec{g}$, where $\vec{u} \in V_p$ and
$\vec{g} \in W^{1,p}(\Omega)$ fulfils the boundary and divergence
data, i.e.~the vector field $\vec{g} \in W^{1,p}(\Omega)$ solves
\begin{equation} \label{eq:g_problem}
\begin{split}
\div \vec{g} &= g_1, \\
\vec{g}_{\vert \del \Omega} &= \vec{g_2}.
\end{split}
\end{equation}
For the corresponding homogeneous system, we have the fundamental
result due to Bogovski\u{\i} (cf.~\cite{bo1}, \cite{bo2}, \cite{Galdi}):
\begin{thm}[Bogovski\u{\i} operator]\label{thm:bogovski}
	Let $\Omega \subset \Rd$ be a bounded Lipschitz domain with
	$d \geq 2$ and $p \in (1,\infty)$. Then there exists a linear and
	bounded operator $\Bog \colon L_0^p(\Omega) \to W_0^{1,p}(\Omega)$
	and a constant $c_{Bog} = c(\Omega, p)$ such that
	\begin{align*}
		\div \Bog f &= f, \\
		\norm{\Bog f}_{1,p} &\leq c_{Bog} \norm{f}_p
	\end{align*}
	for all $f \in L_0^p(\Omega)$. 
\end{thm}

For the inhomogeneous system, we combine Bogovski\u{\i}'s Theorem and
the fact that $ W^{1-\frac{1}{p}, p}(\del \Omega)$ is precisely the
space of boundary values of $W^{1,p}(\Omega)$-functions:
\begin{lem}[The inhomogeneous divergence equation] \label{lem:div_eq}
	Let $\Omega \subset \Rd$ be a bounded Lipschitz domain with
	$d \geq 2$ and $p \in (1, \infty)$. Suppose $g_1 \in L^p(\Omega)$
	and $\vec{g_2} \in W^{1-\frac{1}{p}, p}(\del \Omega)$ satisfy
	$\int_{\Omega} g_1\, dx = \int_{\del \Omega} \vec{g_2} \cdot
	\vec{\nu}\, do$.
	Then there exists a solution $\vec{g} \in W^{1,p}(\Omega)$ of
	problem \eqref{eq:g_problem} that satisfies
	\begin{equation*}
		\norm{\vec{g}}_{1,p} 
		\leq c_{\textit{lift}} \left( 1 + c_{Bog} \right) \norm{\vec{g_2}}_{1-\frac{1}{p}, p} 
		+ c_{Bog} \norm{g_1}_p
	\end{equation*}
	with constants $c_\text{lift}$ and $c_\text{Bog}$ from the trace
	lifting and the Bogovski\u{\i} operator.
\end{lem}
\begin{proof}
	Due to \cite{Gagliardo}, there exists a trace lifting
	$\vec{\hat{g}} \in W^{1,p}(\Omega)$ of the boundary values
	$\vec{g_2}$. By integration by parts, we see that the function
	$g_1 - \div \vec{\hat{g}}$ has mean value zero.  Thus, we may apply the
	Bogovski\u{\i} operator and directly obtain that
	$\vec{g} := \vec{\hat{g}} + \Bog(g_1 - \div \vec{\hat{g}}) \in
	W^{1,p}(\Omega)$ solves \eqref{eq:g_problem}. The estimate of
	$\vec{g}$ follows from the boundedness of the trace lifting and the
	Bogovski\u{\i} operator.
\end{proof}

\subsection{Local coercivity}
We will work with the following notion of local coercivity: 
\begin{defi}[local coercivity]\label{defi:loc_coercive}
	Let $X$ be a Banach space. An operator $A \colon X \to X\d$ is
	called \emph{locally coercive with radius} $R$ if there exists a
	positive real number $R$ such that
	\begin{equation*}
		\fp{Ax}{x} \geq 0
	\end{equation*}
	holds for all $x \in X$ with $\norm{x}_X = R$.
\end{defi}

Local coercivity is precisely the condition that allows to apply
Brouwer's fixed point theorem in order to obtain approximate solutions
in the proof of Br\'ezis' theorem about pseudomonotone operators
\cite[Thm.~27.A]{Zeidler2B}. Therefore, we get a generalized version
of Br\'ezis' theorem that can be proved along the lines of the
standard version. It can also be regarded as a special case of the
existence theorem of Hess and Kato \cite[Thm.~27.B]{Zeidler2B}.

\begin{thm}[Existence theorem for pseudomonotone operators]\label{thm:brezis}
	Let $X$ be a reflexive and separable Banach space and
	$A \colon X \to X^\ast$ be a pseudomonotone, demicontinuous and
	bounded operator that is locally coercive with radius $R$. Then
	there exists a solution $u \in X$ of the problem
	\begin{equation*}
		Au = 0
	\end{equation*}
	that  satisfies $\norm{u}_X \leq R$.
\end{thm}

\subsection{The extra stress tensor and its induced
	operator} \label{sub:extra_stress}

The stress tensor describes the mechanical properties 
of the fluid in dependence on the strain rate $\vec{Dv}$. In Newtonian fluid
dynamics, the viscosity is a constant $\kappa \in \R$ which induces
the linear operator $- \div \S(\vec{Dv}) = - \kappa \Delta \vec{v}$
describing the viscous part of the stress tensor. The general
situation of  non-Newtonian fluids can be modeled in various ways
(cf.~\cite{Saramito}, \cite{BoyerFabrie}).  Here, we consider the
class of fluids with extra stress tensor having
$p$-$\delta$-structure. This class includes and generalizes power law fluids,
where the constitutive relation is given by 
\begin{equation*}
	\S(\vec{Dv}) = \mu_0 \vec{Dv} + \mu (\delta + \abs{\vec{Dv}})^{p-2} \vec{Dv}
\end{equation*}
with material constants $p \in (1, \infty)$,
$\mu_0, \mu, \delta \geq 0$ (cf.~\cite{Cetraro}).

\begin{defi}[extra stress tensor] \label{defi:stress_tensor} An operator
	$\S \colon \mathbb{R}_\textrm{sym}^{d \times d} \to
	\mathbb{R}_{\textrm{sym}}^{d \times d}$ is called an \emph{extra
		stress tensor} with $p$-$\delta$-structure if it is continuous,
	satisfies $\S(\vec{0}) = \vec{0}$ and if there exist constants
	$p \in (1, \infty)$, $\delta \geq 0$ and $C_1(\S)$, $C_2(\S) > 0$
	such that 
	\begin{equation} \begin{split} \label{eq:pdelta_growth1}
	(\S(\vec{A}) - \S(\vec{B}))\cdot (\vec{A}-\vec{B}) &\geq C_1(\S)
	\, (\delta + \abs{\vec{B}} + \abs{\vec{A}-\vec{B}})^{p-2}
	\abs{\vec{A}-\vec{B}}^2,
	\\ 
	\abs{\S(\vec{A}) - \S(\vec{B})} &\leq C_2(\S) \, (\delta +
	\abs{\vec{B}} + \abs{\vec{A}-\vec{B}})^{p-2}
	\abs{\vec{A}-\vec{B}}
	\end{split}\end{equation}
	holds for all $\vec{A}, \vec{B} \in \mathbb{R}_\text{sym}^{d \times
		d}$. The constants $C_1(\S), C_2(\S)$ and $p$ are called the
	characteristics of $\S$. 
\end{defi}

\begin{lem} [{\cite{Cetraro}}] \label{lem:pdelta_growth}
	Let $\S$ be an extra stress tensor with $p$-$\delta$-structure. Then, it holds
	\begin{equation*}
		\fp{\S(\vec{Dv}) - \S(\vec{Dw})}{\vec{Dv}-\vec{Dw}}
		\geq C_3(\S) \int_{\Omega} \int_0^{\abs{\vec{Dv}-\vec{Dw}}} (\abs{\vec{Dw}} + \delta + s)^{p-2}s \, ds\, dx
	\end{equation*}
	for $\vec{v}, \vec{w} \in W^{1,p}(\Omega)$ with a constant $C_3(\S)$
	that only depends on the characteristics of $\S$.
\end{lem}

Since we represented the inhomogeneous data in \eqref{eq:main_problem}
by a fixed function $\vec{g}$ and since we want to solve
\eqref{eq:main_problem} by the ansatz $\vec{v}=\vec{u}+\vec{g}$ with
$\vec{u} \in V_p$, we shall work with a shifted version of the viscous
stress tensor. Therefore, we define the \emph{induced} operator
$\vec{S}\colon W_0^{1,p}(\Omega) \to W_0^{1,p}(\Omega)\d$  via
\begin{equation} \label{eq:defi_S}
\fp{\vec{S}(\vec{v})}{\vec{\phi}}
:= \fp{\S(\vec{Dv}+\vec{Dg})}{\vec{D\phi}}
\end{equation}
for $\vec{v}, \vec{\phi} \in W_0^{1,p}(\Omega)$.

\begin{lem}[Properties of $\vec{S}$] \label{lem:properties_S} Let $\S$
	be an extra stress tensor with $p$-$\delta$-structure,
	$p \in (1, 2]$ and $\vec{g} \in W^{1,p}(\Omega)$. Then the induced
	operator~$\vec{S}$ defined in \eqref{eq:defi_S} is well-defined,
	bounded and continuous.
\end{lem}
\begin{proof}
	Using \eqref{eq:pdelta_growth1}$_2$ with $\vec{A} = \vec{Dw}$ and
	$\vec{B} = \vec{0}$, we obtain
	\begin{equation*}
		\abs{\S(\vec{Dw})}^{p'} 
		\leq C_2(\S)^{p'} \left[ (\abs{\vec{Dw}}+\delta)^{p-2} \abs{\vec{Dw}} \right]^{p'}
		\leq C_2(\S)^{p'} (\abs{\vec{Dw}}+\delta)^p
	\end{equation*}
	and consequently 
	\begin{equation} \label{eq:S_bdd}
	\norm{\S(\vec{Dw})}_{p'} \leq
	C_2(\S) \bignorm{\abs{\vec{Dw}}+\delta}_p^{p-1}
	\end{equation}
	for any $\vec{w} \in W^{1,p}(\Omega)$.  From this, we deduce that
	$\vec{S}\colon W_0^{1,p}(\Omega) \to W_0^{1,p}(\Omega)\d$ is
	well-defined and bounded.
	
	In order to prove continuity, let
	$\vec{v^n} \to \vec{v} \in W_0^{1,p}(\Omega)$ be a convergent
	sequence. Then, by the H\"older inequality and by
	\eqref{eq:pdelta_growth1}$_2$ we get 
	\begin{align*}
		\Vert\vec{S}(\vec{v^n}) - \vec{S}(\vec{v})
		&\Vert_{W_0^{1,p}(\Omega)\d}
		\leq \norm{\S(\vec{Dv^n}+\vec{Dg}) - \S(\vec{Dv}+\vec{Dg})}_{p'} \\
		&\leq C_2(\S) \norm{\left(\delta + \abs{\vec{Dv}+\vec{Dg}} +
			\abs{\vec{Dv^n}-\vec{Dv}}\right)^{p-2}  
			\abs{\vec{Dv^n}-\vec{Dv}}}_{p'} \\
		&\leq C_2(\S) \norm{\vec{Dv^n}-\vec{Dv}}_p^\frac{p}{p'}
		\xrightarrow{n \to \infty} 0.
		\\[-12mm]
	\end{align*}
\end{proof}

Our next goal is to describe coercivity properties of the operator
$\vec{S}$. For the proof of a good lower bound of $\vec{S}$, we
prove an auxiliary algebraic result.

\begin{lem} \label{lem:estim_young_fct}
	Let $a, t \geq 0$ and $p \in (1,2 ] $. Then it holds
	\begin{equation*}
		\int_{0}^{t} (a+s)^{p-2} s \, ds \geq \tfrac{1}{p} t^p - ta^{p-1}.
	\end{equation*}
\end{lem}
\begin{proof}
	The statement becomes trivial if $a = 0$, so we may assume $a >
	0$. For all $s \geq 0$, it holds
	$\frac{a}{(a+s)^{2-p}} \leq \frac{a}{a^{2-p}} = a^{p-1}$.  We
	estimate
	\begin{equation*}
		\frac{s}{(a+s)^{2-p}} = (a+s)^{p-1} - \frac{a}{(a+s)^{2-p}} \geq
		(a+s)^{p-1} - a^{p-1} \geq s^{p-1} - a^{p-1} 
	\end{equation*}
	and by integration we obtain the result.
\end{proof}

With this tool, we are able to prove a lower bound for $\vec{S}$: 
\begin{lem}[Lower bound for $\vec{S}$] \label{lem:S_estimate} For a
	given extra stress tensor $\S$ with $p$-$\delta$-structure,
	$p \in (1, 2 ] $, and a function $\vec{g} \in W^{1,p}(\Omega)$, the
	induced operator
	$\vec{S}\colon W_0^{1,p}(\Omega) \to W_0^{1,p}(\Omega)\d$, defined
	in \eqref{eq:defi_S}, satisfies  the lower bound 
	\begin{equation*}
		\fp{\vec{S}(\vec{v})}{\vec{v}} 
		\geq \tfrac{C_3(\S)}{p} \norm{\vec{Dv}}_p^p - \big ( C_2(\S) +
		C_3(\S) \big) \bignorm{\abs{\vec{Dg}}+\delta}_p^{p-1}
		\norm{\vec{Dv}}_p 
	\end{equation*}
	for all $\vec{v} \in W_0^{1,p}(\Omega)$. 
\end{lem}

\begin{proof}
	We apply Lemma \ref{lem:pdelta_growth} with
	$\vec{v} = \vec{v}+\vec{g}$ and $\vec{w}=\vec{g}$ and Lemma
	\ref{lem:estim_young_fct} to estimate
	\begin{align*}
		\fp{\S(\vec{Dv}+\vec{Dg}) - \S(\vec{Dg})}{\vec{Dv}}
		&\geq C_3(\S) \int_{\Omega} \int_{0}^{\abs{\vec{Dv}}}
		(\abs{\vec{Dg}} + \delta + s)^{p-2} s \, ds \, dx
		\\
		&\geq C_3(\S) \int_{\Omega} \tfrac{1}{p} \abs{\vec{Dv}}^p -
		\abs{\vec{Dv}} (\abs{\vec{Dg}} + \delta)^{p-1} \, dx.
	\end{align*}
	This, the H\"older inequality and \eqref{eq:S_bdd} with
	$\vec{w} = \vec{g}$ 
	\begin{align*}
		\fp{\vec{S}(\vec{v})}{\vec{v}}
		&= \fp{\S(\vec{Dv}+\vec{Dg}) - \S(\vec{Dg})}{\vec{Dv}} +
		\fp{\S(\vec{Dg})}{\vec{Dv}} \nonumber
		\\
		&\geq \tfrac{C_3(\S)}{p} \norm{\vec{Dv}}_p^p 
		- \left[C_3(\S)\, \big\Vert
		\left(\abs{\vec{Dg}}+\delta\right)^{p-1} \big\Vert_{p'} +
		\norm{\S(\vec{Dg})}_{p'}\right] \norm{\vec{Dv}}_p
		\\ 
		&\geq \tfrac{C_3(\S)}{p} \norm{\vec{Dv}}_p^p 
		- \big (C_2(\S) + C_3(\S)\big )
		\bignorm{\abs{\vec{Dg}}+\delta}_p^{p-1}
		\norm{\vec{Dv}}_p , 
	\end{align*}
	which is the assertion.
\end{proof}

In the treatment of the inhomogeneous problem \eqref{eq:main_problem},
we will have to deal with the shifted extra stress tensor
$\vec{A} \mapsto \S(\vec{A} + \vec{G})$ for some constant symmetric
matrix $\vec{G}$. In order to get a precise description of the growth
behavior of this mapping, we introduce the notion of locally uniform
monotonicity:

\begin{defi}[Locally uniform monotonicity] \label{defi:uni_monotone}
	Let $X$ be a reflexive Banach space and $A\colon X \to X\d$ an
	operator.  The operator $A$ is called \emph{locally uniformly
		monotone on $X$}
	if for every $y\in X$ there exists a strictly monotonically
	increasing function $\rho_y\colon [0, \infty) \to [0, \infty)$ with
	$\rho_y(0) = 0$ such that for all $x\in X$ holds
	\begin{equation} \label{eq:uni_monotone}
	\fp{Ax-Ay}{x-y} \geq \rho_y(\norm{x-y}_X)\,.
	\end{equation}
\end{defi}

By the lower bound \eqref{eq:pdelta_growth1}$_1$, we obtain that
(possibly shifted) extra stress tensors are locally uniformly
monotone.

\begin{lem} \label{lem:uni_monotonicity}
	Let $\S\colon \R_{\text{sym}}^{d\times d}\to \R_{\text{sym}}^{d\times
		d}$ be an extra stress tensor with $p$-$\delta$-structure and
	$\vec{G} \in \R_{\text{sym}}^{d\times d}$ be a symmetric
	matrix. Then the shifted extra stress tensor
	$\vec{A} \mapsto \S(\vec{A}+\vec{G})$ is a locally uniformly
	monotone operator on $\R_{\text{sym}}^{d\times d}$.
\end{lem}
\begin{proof}
	By \eqref{eq:pdelta_growth1}$_1$, we obtain for any $\vec{A},
	\vec{B} \in \R_{\text{sym}}^{d\times d}$ 
	\begin{align*}
		&(\S(\vec{A}\!+\!\vec{G}) \! -\! \S(\vec{B}\!+\!\vec{G})) \cdot
		(\vec{A}\!-\!\vec{B})
		\geq C_1(\S)  \big(\delta \!+\! \abs{\vec{B}\!+\!\vec{G}}\! +\!
		\abs{\vec{A}\!-\!\vec{B}}\big)^{p-2} \abs{\vec{A}\!-\!\vec{B}}^2.
	\end{align*}
	For any $\vec{B}$, the function
	$\rho_{\vec{B}}(t) := C_1(\S) (\delta + \abs{\vec{B}+\vec{G}} +
	t)^{p-2} t^2$ is non-negative, satisfies $\rho_{\vec{B}}(0) = 0$,
	and it is strictly monotonically increasing since for its derivative
	it holds
	\begin{equation*}
		\rho'_{\vec{B}}(t) 
		= C_1(\S) (\delta + \abs{\vec{B}+\vec{G}} + t)^{p-3} t (2\delta +
		2\abs{\vec{B}+\vec{G}} + pt) > 0 
	\end{equation*}
	for all $t > 0$. Therefore, it fulfils the requirements from
	Definition \ref{defi:uni_monotone}.
\end{proof}

\subsection{Properties of the convective term} \label{sub:conv_term}

Since we fixed a function $\vec{g}$ that expresses the inhomogeneous
data in \eqref{eq:main_problem}, we shall work with a "shifted"
version of the convective term
$\fp{(\vec{u}+\vec{g}) \, \cdot \nabla (\vec{u}+\vec{g})}{\vec{\phi}}$
that is integrable and thus well-defined even for
$p > \tfrac{2d}{d+2}$ and sufficiently regular $\vec{\phi}$ and
$\vec{g}$.  Therefore, we set
\begin{equation}\label{eq:s}
s=s(p) := \max \Big\{ p, \Big(\frac{p\d}{2}\Big)' \Big\}
= \left\{ \begin{array}{cl}
p &\text{if }p > \frac{3d}{d+2},
\\
\left(\frac{p\d}{2}\right)' &\text{if } p \leq \frac{3d}{d+2}
\end{array} \right.
\end{equation}
for $p \in \big(\tfrac{2d}{d+2}, 2\big)$ and define the convective
term $\vec{T}\colon V_p \to W_0^{1,s}(\Omega)\d$ via
\begin{equation} \label{eq:defi_T}
\fp{\vec{T}(\vec{u})}{\vec{\phi}}
:= - \fp{(\vec{u}+\vec{g}) \otimes (\vec{u}+\vec{g})}{\vec{D\phi}} -
\fp{(\div \vec{g})(\vec{u}+\vec{g})}{\vec{\phi}}
\end{equation}
for $\vec{u} \in V_p$ and $\vec{\phi} \in W_0^{1,s}(\Omega)$.

\begin{lem}[Properties of the convective term] \label{lem:t2_str_contin}
	For $p \in \big(\frac{2d}{d+2}, 2\big)$ let $s $ be defined in
	\eqref{eq:s} 
	and let $q \in \R$ satisfy $q \geq s$ and $q > (\frac{p\d}{2})'$. Then, for any given function
	$\vec{g} \in W^{1,s}(\Omega)$, the operator $\vec{T}$ defined in 
	\eqref{eq:defi_T} is formally equivalent to
	$\fp{(\vec{u}+\vec{g}) \, \cdot \nabla
		(\vec{u}+\vec{g})}{\vec{\phi}}$. It is well-defined and bounded
	from $V_p$ to $W_0^{1,s}(\Omega)\d$ and also from $V_p$ to
	$W_0^{1,q}(\Omega)\d$. The operator  $\vec{T}$ is continuous from $V_p$ to
	$W_0^{1,s}(\Omega)\d$ and strongly continuous from
	$V_p$ to $W_0^{1,q}(\Omega)\d$.  It fulfils the estimate
	\begin{align}\label{eq:econv}
		\begin{split}
			\abs{\fp{\vec{T}(\vec{u})}{\vec{u}}} &\leq c_\text{Sob}\,
			c_\text{Korn}^2 \left( \norm{ \vec{Dg}}_s + \tfrac{1}{2}
			\norm{\div \vec{g}}_s \right) \norm{\vec{Du}}_p^2
			\\
			&\quad + c_\text{Sob}\, \big( \norm{\vec{g}}_{1,s}^2 +
			c_\text{Korn} \norm{\div \vec{g}}_s \norm{\vec{g}}_{1,s} \big)
			\norm{\vec{Du}}_p
		\end{split}
	\end{align}
	for all $\vec{u} \in V_q$, where
	$c_\text{Sob}$ are Sobolev embedding constants and
	$c_\text{Korn}$ is the constant in the  Korn inequality for $\Omega$.
\end{lem}

\begin{proof} 
	The formal equivalence follows from a straightforward computation
	with integration by parts.  We abbreviate $g_1 := \div \vec{g} \in
	L^s(\Omega)$ and use the continuous Sobolev embeddings
	$W^{1,s}(\Omega) \injto W^{1,p}(\Omega) \injto
	L^{p\d}(\Omega)$. The definition of $s$ implies $\frac{1}{p\d} +
	\frac{1}{p\d}+ \frac{1}{s} \leq
	1$, so both well-definedness of $\vec{T}(\vec{u}) \in
	W_0^{1,s}(\Omega)\d$ for $\vec{u} \in
	V_p$ and boundedness follow by the H\"older inequality.
	
	In view of the continuous embedding
	$W_0^{1,s}(\Omega)\d \injto W_0^{1,q}(\Omega)\d$, we immediately
	obtain well-definedness and boundedness if $\vec{T}$ is considered
	as an operator from $V_p$ to $W_0^{1,q}(\Omega)\d$.
	
	Since $\frac{1}{p\d} + \frac{1}{p\d}+ \frac{1}{q} < 1$, there is
	some $\tau < p\d$ such that
	$\frac{1}{\tau} + \frac{1}{p\d}+ \frac{1}{q} = 1$. Let
	$\vec{u^n} \weakto \vec{u} \in V_p$~be a weakly convergent
	sequence. The Sobolev embedding
	$W^{1,p}(\Omega) \injto\injto L^{\tau}(\Omega)$ is compact, so
	$\vec{u^n} \to \vec{u} \in L^{\tau}(\Omega)$ converges strongly.
	Thus, we estimate
	\begin{align*}
		&\sup_{\norm{\vec{D\phi}}_q \leq 1} 
		\abs{\fp{(\vec{u^n} +\vec{g}) \otimes (\vec{u^n}+\vec{g}) 
				- (\vec{u}+\vec{g}) \otimes (\vec{u}+\vec{g})}{\vec{D\phi}}} 
		\\
		&\quad = \sup_{\norm{\vec{D\phi}}_q \leq 1} \big \vert \langle\vec{u^n} \otimes (\vec{u^n}-\vec{u}) 
		+ (\vec{u^n}-\vec{u}) \otimes \vec{u} \\
		&\qquad\qquad\qquad 	+ \vec{g} \otimes (\vec{u^n}-\vec{u}) 
		+ (\vec{u^n}-\vec{u}) \otimes \vec{g}, \vec{D\phi}\rangle \big \vert 
		\\
		&\quad \leq \norm{\vec{u^n}}_{p\d} \norm{\vec{u^n}-\vec{u}}_\tau 
		+ \norm{\vec{u^n}-\vec{u}}_\tau \norm{\vec{u}}_{p\d}
		+ 2 \norm{\vec{g}}_{p\d} \norm{\vec{u^n}-\vec{u}}_\tau
		\xrightarrow{\,n \to \infty\,} 0.
	\end{align*}	
	Similarly, we obtain
	\begin{equation*}
		\sup_{\norm{\vec{D\phi}}_q \leq 1} \abs{\fp{g_1(\vec{u^n}-\vec{u})}{\vec{\phi}}}
		\leq C \norm{g_1}_s \norm{\vec{u^n}-\vec{u}}_\tau
		\xrightarrow{n \to \infty} 0.
	\end{equation*}
	Thus, we proved $\vec{T}(\vec{u^n}) \to \vec{T}(\vec{u})$ in
	$W_0^{1,q}(\Omega)\d$, i. e.
	$\vec{T}\colon V_p \to W_0^{1,q}(\Omega)\d$ is strongly continuous.
	
	Analogously, we prove continuity for
	$\vec{T} \colon V_p \to W_0^{1,s}(\Omega)\d$ using
	$\vec{u^n} \to \vec{u} \in V_p$ and the continuous embedding
	$W^{1,p}(\Omega) \injto L^{p\d}(\Omega)$.
	
	For the bound \eqref{eq:econv} of $\vec{T}$, we use 
	\begin{equation} \label{eq:t2_eq1}
	\fp{\vec{u}\otimes\vec{u}}{\vec{Du}} = 0
	\end{equation}
	which follows by integration by parts, since $\vec{u}$ has zero
	divergence and zero boundary values.  In the same way, we see
	\begin{equation} \label{eq:t2_eq2}
	\fp{\vec{u} \otimes
		\vec{g}}{\nabla \vec{u}} = - \fp{\vec{u} \otimes
		\vec{u}}{\vec{Dg}}
	\end{equation}
	and
	\begin{equation} \label{eq:t2_eq3}
	\fp{\vec{u} \otimes
		\vec{g}}{\nabla \vec{u}^T} = -\tfrac{1}{2}\, \fp{g_1
		\vec{u}}{\vec{u}}.
	\end{equation}
	Using \eqref{eq:t2_eq1}, \eqref{eq:t2_eq2} and \eqref{eq:t2_eq3}
	in the definition of $\vec{T}$, we obtain the following expression
	for the convective term:
	\begin{align}
		\langle \vec{T} (\vec{u}), \vec{u}\rangle
		&= \fp{\vec{u} \otimes \vec{u}}{\vec{Du}} + \fp{\vec{u} \otimes
			\vec{g}}{\nabla \vec{u}} 
		+ \fp{\vec{u} \otimes \vec{g}}{\nabla \vec{u}^T} + \fp{\vec{g}
			\otimes \vec{g}}{\vec{Du}} \nonumber
		\\
		&\quad + \fp{g_1 \vec{u}}{\vec{u}} + \fp{g_1 \vec{g}}{\vec{u}}
		\nonumber
		\\
		&= - \fp{\vec{u} \otimes \vec{u}}{\vec{Dg}} + \tfrac{1}{2}\,
		\fp{g_1 \vec{u}}{\vec{u}}   + \fp{\vec{g} \otimes
			\vec{g}}{\vec{Du}} + \fp{g_1\vec{g}}{\vec{u}}. \label{eq:t2_eq4} 
	\end{align}
	In order to estimate this expression, we use the Sobolev embedding
	$W^{1,s}(\Omega) \injto L^{2p'}(\Omega)$. In fact, in the case
	$s \geq d$ this follows directly. If $\frac{3d}{d+2} < p < 2$ and
	$s<d$, we have $s\d = p\d > 2p'$, due to a straightforward
	computation. If $\frac{2d}{d+1} < p \leq \frac{3d}{d+2}$ and
	$s < d$, we get
	$s\d \geq \big(\big(\frac{p\d}{2}\big)'\big)\d = \frac{pd}{pd-2d+p}
	\geq 2p'$. Finally, if $\frac{2d}{d+2} < p \leq \frac{2d}{d+1}$,
	then it holds $s \geq \big(\frac{p\d}{2}\big)' \geq d$.  Applying
	the H\"older and the Korn inequality and the embeddings
	$W^{1,s}(\Omega) \injto L^{2p'}(\Omega)$ and
	$W^{1,p}(\Omega) \injto L^{p\d}(\Omega)$ to \eqref{eq:t2_eq4}, the 
	claimed estimate follows.
\end{proof}

\subsection{Lipschitz truncation} \label{sub:lip_trunc}
In case of a small growth parameter $p$, this means
$p \in \big( \frac{2d}{d+2}, \frac{3d}{d+2} \big)$, a function
$\vec{v} \in W^{1,p}(\Omega)$ does not have enough integrability to be
chosen as a test function in operators like
$\fp{\vec{v}\otimes \vec{v}}{\vec{D\phi}}$.  Hence, we use
sufficiently smooth approximations of the test functions in the limit
process of the existence proof, which are given by the Lipschitz
truncation method. The existence of Lipschitz truncations is
guaranteed by the following result proved in \cite{fms2}, \cite{dms}, \cite{Cetraro}:
\begin{thm}[Lipschitz truncation] \label{thm:lip_trunc}
	Let $\Omega$ be a bounded domain with Lipschitz continuous
	boundary, let $p \in (1, \infty)$ and let 
	$(\vec{v^n})_{n \in \N} \subset W_0^{1, p}(\Omega)$ be a sequence
	such that   $\vec{v^n} \weakto \vec{0}$ weakly.
	
	Then, for all $j, n \in \N$, there exists a function
	$\vec{v_j^n} \in W_0^{1, \infty}(\Omega)$ and a number
	$\lambda_j^n \in \big[ 2^{2^j}, 2^{2^{j+1}} \big]$ such that
	\begin{align} \label{eq:lip_trunc1} \begin{split}
			\lim\limits_{n \to \infty} \big( \sup\limits_{j \in \N}
			\norm{\vec{v_j^n}}_{\infty} \big) &= 0,
			\\ 
			\norm{\nabla \vec{v_j^n}}_{\infty} &\leq c \, \lambda_j^n,
			\\
			\limsup_{n \to \infty} \left( \lambda_j^n \right)^p \abs{\{
				\vec{v_j^n} \neq \vec{v^n} \}} &\leq c \, 2^{-j},
			\\ 
			\limsup_{n \to \infty} \norm{\nabla \vec{v_j^n} \, \chi_{\{
					\vec{v_j^n} \neq \vec{v^n} \}}}_p^p &\leq c \, 2^{-j}
	\end{split} \end{align}
	holds with a uniform constant $c = c(d, p, \Omega)$. 
	
	Moreover, for fixed $j \in \N$ and $r \in [1, \infty)$, we have
	\begin{align} \label{eq:lip_trunc2}
		\begin{split}
			\nabla \vec{v_j^n} \weakto \vec{0} &\quad \text{ in }
			L^r(\Omega),
			\\
			\nabla \vec{v_j^n} \weakstarto \vec{0} &\quad \text{ in }
			L^{\infty}(\Omega) 
		\end{split}
	\end{align}
	as $n \to \infty$.
\end{thm}

The following lemma shows how Lipschitz truncation can be used to get
a connection between weak and almost everywhere convergence. Statement
and proof are close to \cite[Lemma 2.6]{dms}, only the
assumptions on the operator $\S$ have been reduced for the reasons
discussed in the previous Subsection~\ref{sub:extra_stress}.

\begin{lem}[Almost everywhere convergence for the Lipschitz truncation
	method] \label{lem:ae_equality2} Let $\Omega$ be a bounded domain,
	$p \in (1, \infty)$,
	$(\vec{u^n})_{n \in \N} \subset W^{1,p}(\Omega)$ be a weakly
	convergent sequence with limit $\vec{u} \in W^{1,p}(\Omega)$.  Let
	$\A\colon \R_{\text{sym}}^{d\times d}\to \R_{\text{sym}}^{d \times
		d}$ be a locally uniformly monotone operator on $\R_{\text{sym}}^{d\times d}$ such that the induced
	operator $\vec{w} \mapsto \A(\vec{Dw})$ is well-defined and
	bounded in $W^{1,p}(\Omega) \to L^{p'}(\Omega)$.
	
	Now let $B$ be a ball with $2B \subset \subset \Omega$ and
	$\xi \in C_0^\infty (\Omega)$ be a cutoff function such that 
	$\chi_B \leq \xi \leq \chi_{2B}$. We set
	$\vec{v^n} := \xi (\vec{u^n}- \vec{u})$ and let $\vec{v_j^n}$ be the
	Lipschitz truncation of $\vec{v^n}$ with respect to the domain $2B$
	as described in Theorem \ref{thm:lip_trunc}. If we have
	\begin{equation} \label{eq:ae_equality_assum}
	\lim_{n \to \infty} \abs{\fp{\A(\vec{Du^n}) -
			\A(\vec{Du})}{\vec{Dv_j^n}}} \leq \delta_j 
	\end{equation}
	for all $j \in \N$ and some sequence $(\delta_j)_{j \in \N}$ with
	$\lim_{j \to \infty} \delta_j = 0$, then a subsequence of
	$\vec{Du^n}$ converges to $\vec{Du}$ almost everywhere in $B$.
\end{lem}

\begin{rem}\label{rem:S}
	In Lemma \ref{lem:uni_monotonicity} and lemma \ref{lem:properties_S}
	we have seen that for an extra stress tensor
	$\S\colon \R_{\text{sym}}^{d\times d}\to \R_{\text{sym}}^{d\times d}$
	with $p$-$\delta$-structure and a given vector field
	$\vec {g} \in W^{1,p}(\Omega)$ the operator
	$\A(\vec{B}):=\S(\vec{B} + \vec{Dg})$ fulfils the requirements in
	Lemma~\ref{lem:ae_equality2}.
\end{rem}

\begin{proof}[Proof of Lemma \ref{lem:ae_equality2}]
	Let $\theta \in (0, 1)$. Making use of the properties of $\A$, we
	obtain strong convergence
	$\left[(\A(\vec{Du^n})-\A(\vec{Du}))\cdot
	(\vec{Du^n}-\vec{Du})\right] ^\theta \to 0$ in $L^1(B)$ as
	$n \to \infty$ along the lines of \cite[Lemma 2.6]{dms}. We switch to a
	subsequence that converges almost everywhere.
	By the definition of locally uniform monotonicity, there exists a
	strictly monotonically increasing function
	$\rho_x\colon [0, \infty) \to [0, \infty)$ with
	\begin{align*}
		&(\A(\vec{Du^n}(x))-\A(\vec{Du}(x)))\cdot
		(\vec{Du^n}(x)-\vec{Du}(x))
		\geq \rho_x   (\abs{\vec{Du^n}(x)-\vec{Du}(x)}) 
	\end{align*}
	for all $n\in\N$ and almost every $x\in B$ ($\rho_x$~depends on
	$\vec{Du}(x)$). Utilizing the almost everywhere convergence of the
	left-hand side and the non-negativity of the right-hand side, we
	obtain a subsequence that fulfils
	$\rho_x (\abs{\vec{Du^n}(x)-\vec{Du}(x)}) \to 0$ almost
	everywhere. Thus, it holds $\vec{Du^n}(x) \to \vec{Du}(x)$ as
	$n\to\infty$ for this subsequence and almost every~$x\in B$.	
\end{proof}

By applying a covering argument and taking the diagonal sequence we
obtain a global version of Lemma \ref{lem:ae_equality2}
(cf.~\cite[Cor. 3.32]{Cetraro}):

\begin{cor}\label{cor:ae_equality3}
	Assume that the assumptions of Lemma \ref{lem:ae_equality2} are fulfilled
	for all balls $B$ with $2B \subset \subset \Omega$ (with sequences
	$\delta_j$ that may depend on the ball $B$). Then $\vec{Du^n}$
	converges to $\vec{Du}$ almost everywhere on $\Omega$ for a suitable
	subsequence.
\end{cor}

Using the almost everywhere convergence established in Corollary
\ref{cor:ae_equality3}, we may prove a general statement about the
limit process with the Lipschitz truncation method in existence
proofs:

\begin{thm}[Identification of limits using the Lipschitz truncation
	method] \label{thm:ident_limits} Let $\Omega \subset \Rd$ be a
	bounded domain, $p \in (1, \infty)$ and $\vec{g}\in W^{1,p}(\Omega)$
	be given. Let
	$\A\colon \R_{\text{sym}}^{d\times d}\to \R_{\text{sym}}^{d \times
		d}$ be a continuous, locally uniformly monotone operator on
	$\R_{\text{sym}}^{d\times d}$ such that the induced operator
	$\vec{w} \mapsto \A(\vec{Dw})$ is well-defined and bounded
	in $W^{1,p}(\Omega) \to L^{p'}(\Omega)$. Let there be an operator
	$\vec{B} \colon V_p \to V_s\d$ for some $s \in [p, \infty)$ and a
	space $X$ such that $X \injto V_p$ embeds continuously and such that
	$\vec{B}$ is well-defined as an operator $X \to X\d$.  Assume we
	have a sequence of operators $\vec{A_n}\colon X \to X\d$ and
	solutions $\vec{u^n} \in X$ to
	\begin{equation} \label{eq:lem_ident_appr}
	\fp{\A(\vec{Du^n})}{\vec{D\phi}} +
	\fp{\vec{B}(\vec{u^n})}{\vec{\phi}} +
	\fp{\vec{A_n}(\vec{u^n})}{\vec{\phi}} = 0
	\end{equation}
	with test functions $\vec{\phi} \in X\!$.
	
	In addition, assume that for some $r \in (1, \infty)$ the embedding
	$V_r \injto X$ is   continuous and dense, that $\vec{B}$ is strongly
	continuous as an operator $V_p \to W_0^{1,r}(\Omega)\d$ and that we have
	convergences
	\begin{align} \label{eq:ident_assum}
		\begin{split}
			\vec{u^n} \weakto \vec{u} &\quad \text{weakly in } V_p,\\
			\vec{A_n}(\vec{u^n}) \to \vec{0}  &\quad \text{strongly in } W_0^{1,r}(\Omega)\d
		\end{split}
	\end{align}
	as $n \to \infty$.
	
	Then $\vec{u}$ is a solution of the limit equation
	\begin{equation} \label{eq:lem_ident_limit}
	\fp{\A(\vec{Du})}{\vec{D\phi}} + \fp{\vec{B}(\vec{u})}{\vec{\phi}} = 0
	\end{equation}
	for all $\vec{\phi} \in V_s$.
\end{thm}

\begin{rem}
	The operator $\A$ represents a (possibly shifted) extra stress
	tensor (cf.~Remark \ref{rem:S}) and
	$\vec{B}$ may be chosen as the convective term. Typical choices for
	the space $X$ are $X = V_q$ or $X = V_p \cap L^q(\Omega)$ with
	coercive operators
	$\fp{\vec{A_n}(\vec{v})}{\vec{\phi}} = \fp{\abs{ \vec{Dv}}^{q-2}
		\vec{Dv}}{\vec{D\phi}}$ and
	$\fp{\vec{A_n}(\vec{v})}{\vec{\phi}} = \fp{\abs{\vec{v}}^{p-2}
		\vec{v}}{\vec{\phi}}$ respectively.
	
	The inclusions $X\injto V_p$, $V_s \injto V_p$ and $V_r \injto X$
	guarantee the well-definedness of $\vec{A_n}$ and of the operator
	which is induced by $\A$.
\end{rem}

\begin{proof}[Proof of Theorem \ref{thm:ident_limits}]
	The proof of Theorem \ref{thm:ident_limits} follows and generalizes
	the procedure in \cite{dms}, \cite{Cetraro}. First, we check the assumptions of
	Lemma \ref{lem:ae_equality2}/ Corollary \ref{cor:ae_equality3} in
	order to obtain almost everywhere convergence
	$\vec{Du^n} \to \vec{Du}$, then we use this to prove
	\eqref{eq:lem_ident_limit}.
	
	As in Lemma \ref{lem:ae_equality2} we let $B$ be a ball with
	$2B \subset \subset \Omega$ and $\xi \in C_0^\infty (\Omega)$ be a
	cutoff function such that $\chi_B \leq \xi \leq \chi_{2B}$. We set
	$\vec{v^n} := \xi (\vec{u^n}- \vec{u})$ and let $\vec{v_j^n}$ be the
	Lipschitz truncation of $\vec{v^n}$ with respect to the domain $2B$
	from Theorem \ref{thm:lip_trunc}. Since the functions $\vec{v_j^n}$
	are in general not divergence-free, we have to introduce correction
	terms in order to use them as test functions in
	\eqref{eq:lem_ident_appr}. We use the Bogovski\u{\i} operator
	$\Bog \colon L_0^r(\Omega) \to W_0^{1,r}(\Omega)$ and set
	\begin{equation} \label{eq:defi_eta}
	\vec{\psi_j^n} := \Bog(\div \vec{v_j^n}) \in W_0^{1,r}(\Omega)
	\quad\text{and}\quad \vec{\eta_j^n} := \vec{v_j^n} -
	\vec{\psi_j^n} \in V_r. 
	\end{equation}
	By \eqref{eq:lip_trunc2}$_1$, we get
	$\nabla \vec{v_j^n} \weakto \vec{0}$ in $L^r(\Omega)$ for each
	$j \in \N$ as $n\to\infty$. Since both the divergence and the
	Bogovski\u{\i} operator are linear and continuous, we get the
	convergence
	\begin{align} \label{eq:conv_psi1}
		\vec{\psi_j^n} \weakto \vec{0} &\quad \text{ weakly in } W_0^{1,r}(\Omega)
	\end{align}
	as $n \to \infty$ for every $j \in \N$.  By a well-known fact, we
	know $\nabla \vec{v_j^n} = \nabla \vec{v^n}$ on the set
	$\{\vec{v_j^n} = \vec{v^n}\}$ (cf.~\cite{Maly}). Thus, we obtain
	$\div \vec{v^n} = \nabla \xi \cdot (\vec{u^n}-\vec{u})$ by the
	product rule and get
	$\div\vec{v_j^n} = \chi_{\{\vec{v^n}\neq\vec{v_j^n} \}} \div
	\vec{v_j^n} + \chi_{\{\vec{v^n}=\vec{v_j^n} \}} \nabla \xi \cdot
	(\vec{u^n}-\vec{u})$. Together with the continuity of the
	Bogovski\u{\i} operator and the $W^{1, \infty}(\Omega)$-boundedness
	of the cutoff function~$\xi$, this implies
	\begin{equation} \label{eq:estim_psi1}
	\norm{\vec{\psi_j^n}}_{1,p}
	\leq c \norm{\div \vec{v_j^n}}_p \leq c
	\big\Vert\chi_{\{\vec{v^n}\neq\vec{v_j^n} \}} \nabla \vec{v_j^n}
	\big\Vert_p + c\,(\xi) \norm{\vec{u^n}-\vec{u}}_p\!.
	\end{equation}
	Furthermore, due to the assumption $\vec{u^n} \weakto \vec{u}$ in
	$W_0^{1,p}(\Omega)$ and the compact embedding
	$W_0^{1,p}(\Omega) \injto \injto L^p(\Omega)$, we have strong
	convergence $\vec{u^n} \to \vec{u}$ in $L^p(\Omega)$. Applying
	\eqref{eq:lip_trunc1}$_4$ and this strong convergence in
	\eqref{eq:estim_psi1}, we obtain
	\begin{equation} \label{eq:estim_psi2}
	\limsup_{n \to \infty} \norm{\vec{\psi_j^n}}_{1,p} \leq c\, 2^{\frac{-j}{p}}
	\end{equation}
	for all $j \in \N$.
	
	From \eqref{eq:lip_trunc2}$_1$, \eqref{eq:conv_psi1} and the compact
	embedding $W_0^{1,r}(\Omega) \injto \injto L^r(\Omega)$, we conclude
	\begin{align} \label{eq:conv_eta1}
		\vec{\eta_j^n} \weakto \vec{0}
		&\quad \text{ weakly in } W_0^{1,r}(\Omega)
	\end{align}
	for all $j \in \N$ as $n \to \infty$.
	
	Since $\vec{B}\colon V_p \to W_0^{1,r}(\Omega)\d$ is strongly
	continuous and $\vec{u^n} \weakto \vec{u}$ in $V_p$, we obtain the
	convergence $\vec{B}(\vec{u^n}) \to \vec{B}(\vec{u})$ in
	$W_0^{1,r}(\Omega)\d$. This and \eqref{eq:conv_eta1} imply
	\begin{equation} \label{eq:lip2}
	\lim_{n \to \infty} \fp{\vec{B}(\vec{u^n})}{\vec{\eta_j^n}} = 0.
	\end{equation}
	Similarly, we obtain
	\begin{equation} \label{eq:lip3}
	\lim_{n \to \infty} \fp{\vec{A_n}(\vec{u^n})}{\vec{\eta_j^n}} = 0
	\end{equation}
	from \eqref{eq:ident_assum}$_2$ and \eqref{eq:conv_eta1}.
	Furthermore, \eqref{eq:lip_trunc2}$_1$ implies
	$\vec{v_j^n} \weakto \vec{0}$ in $W_0^{1,p}(\Omega)$ and
	\begin{equation} \label{eq:lip4}
	\lim_{n \to \infty} \fp{\A(\vec{Du})}{\vec{Dv_j^n}} 
	= 0
	\end{equation}
	for all $j \in \N$.
	
	By \eqref{eq:defi_eta} and equation
	\eqref{eq:lem_ident_appr} we have
	\begin{align*}
		&\langle \A (\vec{Du^n}) - \A(\vec{Du}), \vec{Dv_j^n} \rangle
		\\
		&= - \fp{\vec{B}(\vec{u^n})}{\vec{\eta_j^n}} - \fp{\vec{A_n}(\vec{u^n})}{\vec{\eta_j^n}} 
		+ \fp{\A(\vec{Du^n})}{\vec{D\psi_j^n}} - \fp{\A(\vec{Du})}{\vec{Dv_j^n}}.
	\end{align*}
	We use the convergences \eqref{eq:lip2}, \eqref{eq:lip3},
	\eqref{eq:estim_psi2} and \eqref{eq:lip4} in this identity and obtain
	\begin{equation*}
		\limsup_{n \to \infty} \abs{\fp{\A(\vec{Du^n}) -
				\A(\vec{Du})}{\vec{Dv_j^n}}} \leq c\, 2^{\frac{-j}{p}}\!. 
	\end{equation*}
	Since $2^{\frac{-j}{p}} \to 0$ as $j \to \infty$, we may apply
	Corollary \ref{cor:ae_equality3} and conclude
	$\vec{Du^n} \to \vec{Du}$ almost everywhere in $\Omega$ up to some
	subsequence. By the continuity of $\A$, it follows
	$\A(\vec{Du^n}) \to \A(\vec{Du})$ almost everywhere in $\Omega$.
	
	By assumption, the mapping $\vec{v} \mapsto \A(\vec{Dv})$ defines a
	bounded operator \linebreak$W^{1,p}(\Omega) \to L^{p'}(\Omega)$,
	and thus the sequence 
	$(\A(\vec{Du^n}))_{n \in \N}$ is bounded. We may extract a weakly convergent
	subsequence $\A(\vec{Du^n}) \weakto \vec{\chi}$
	in~$L^{p'}(\Omega)$. The combination of almost everywhere
	convergence $\A(\vec{Du^n}) \to \A(\vec{Du})$ and weak convergence
	$\A(\vec{Du^n}) \weakto \vec{\chi}$ (for some subsequences) implies
	$\A(\vec{Du}) = \vec{\chi}$ by a well-known convergence principle
	(cf.~\cite{Gajewski}); in particular, it follows
	\begin{equation} \label{eq:conv_s}
	\A(\vec{Du^n}) \weakto \A(\vec{Du}) \quad \text{weakly in } L^{p'}(\Omega).
	\end{equation}
	We pass to the limit for $n \to \infty$ in
	\eqref{eq:lem_ident_appr} and use \eqref{eq:conv_s}, the strong
	continuity of $\vec{B}$ and \eqref{eq:ident_assum}$_2$ to obtain
	\begin{align*}
		0 &= \lim_{n \to \infty} \fp{\A(\vec{Du^n})}{\vec{D\phi}} +
		\fp{\vec{B}(\vec{u^n})}{\vec{\phi}} +
		\fp{\vec{A_n}(\vec{u^n})}{\vec{\phi}}
		\\ 
		&= \fp{\A(\vec{Du})}{\vec{D\phi}} + \fp{\vec{B}(\vec{u})}{\vec{\phi}}
	\end{align*}
	for $\vec{\phi} \in V_p \cap V_r \cap X = V_r$ and therefore, by
	density, for all $\vec{\phi} \in V_s$.
\end{proof}

\section{Existence of weak solutions} \label{sec:existence}

\subsection{Smallness condition and main result}
As mentioned in the introduction, our ansatz for proving existence requires smallness of the boundary and the divergence data which is necessary for proving local coercivity. In order to formulate a precise smallness condition, we define the following dependent constants:

For a domain $\Omega$, an extra stress tensor $\S$ with
$p$-$\delta$-structure,  $s = \max \big\{ p, \big (\frac{p\d}{2}\big )' \big\}$, a functional $\vec{f} \in W_0^{1,p}(\Omega)\d$
and a function $\vec{g} \in W^{1,s}(\Omega)$ we define 
\begin{equation} \begin{split} \label{eq:g123} G_1 &:= \tfrac{1}{p} \,
C_3(\S),
\\
G_2 &:= c_\text{Sob} c_\text{Korn}^2 \left[ \norm{\vec{Dg}}_{s} +
\tfrac{1}{2} \norm{\div \vec{g}}_s \right],
\\
G_3 &:= \left( C_2(\S) + C_3(\S) \right)
\bignorm{\abs{\vec{Dg}}+\delta}_p^{p-1} + c_\text{Sob}
\norm{\vec{g}}_{1,s}^2
\\
&\qquad + c_\text{Sob} c_\text{Korn} \norm{\div \vec{g}}_s
\norm{\vec{g}}_{1,s} + c_\text{Korn}
\norm{\vec{f}}_{W_0^{1,p}(\Omega)\d}
\end{split} \end{equation} with constants $c_\text{Korn}$,
$c_\text{Sob}$, $C_2(\S)$ and $C_3(\S)$ that do only depend on
$\Omega$ and the characteristics of~$\S$.

With these constants, we impose a smallness condition on the data
$g_1$ and $\vec{g_2}$:

\begin{assum} \label{assum:smallness}
	We assume that $g_1 \in L^s(\Omega)$ and
	$\vec{g_2} \in W^{s-\frac{1}{s}, s}(\del \Omega)$ satisfy the
	compatibility condition
	$\int_{\Omega} g_1\,dx = \int_{\del \Omega} \vec{g_2} \cdot
	\vec{\nu}\, do$
	and that their norms are so small that a solution
	$\vec{g} \in W^{1,s}(\Omega)$ of the corresponding inhomogeneous
	divergence equation (see Lemma \ref{lem:div_eq}) satisfies
	\begin{equation} \label{eq:g_smallness}
	(2-p)^{2-p} (p-1)^{p-1} G_1 \geq G_2^{p-1}G_3^{2-p}
	\end{equation}
	for the constants $G_1, G_2, G_3$ from \eqref{eq:g123}.
\end{assum}

Under that condition, we are able to prove the following existence result:

\begin{thm}[Existence] \label{thm:main_thm} Let $\Omega \subset \Rd$
	be a bounded Lipschitz domain with $d \in \{2, 3\}$. Let~$\S$ be an
	extra stress tensor with $p$-$\delta$-structure,
	$p \in \big( \frac{2d}{d+2}, 2 \big)$,
	$s := \max \big\{ p, \big (\frac{p\d}{2}\big )' \big\}$ and
	$\vec{f} \in W_0^{1,p}(\Omega)\d$. For any $g_1 \in L^s(\Omega)$ and
	$\vec{g_2} \in W^{s-\frac{1}{s}, s}(\del \Omega)$ that satisfy
	Assumption \ref{assum:smallness}, there exists a weak solution
	$(\vec{v}, \pi) \in W^{1,p}(\Omega) \times L^{s'}(\Omega)$
	of~\eqref{eq:main_problem}.
\end{thm}

\subsection{Existence proof} \label{sub:existence}
To get a formulation of \eqref{eq:main_problem}$_1$, we use the
definitions \eqref{eq:defi_S} and \eqref{eq:defi_T} of the operators
$\vec{S}$ and $\vec{T}$ and define the "full" operator
$\vec{P} \colon V_p \to W_0^{1,s}(\Omega)\d$ via
\begin{equation} \begin{split} \label{eq:defi_P}
\fp{\vec{P}(\vec{v})}{\vec{\phi}} := \,
&\fp{\vec{S}(\vec{v}) + \vec{T}(\vec{v}) - \vec{f}}{\vec{\phi}} \\
= \,&\fp{\S(\vec{Dv}+\vec{Dg})}{\vec{D\phi}} - \fp{(\vec{v}+\vec{g}) \otimes (\vec{v}+\vec{g})}{\vec{D \phi}} \\
&- \fp{(\div \vec{g}) (\vec{v}+\vec{g})}{\vec{\phi}}
- \fp{\vec{f}}{\vec{\phi}}
\end{split} \end{equation}
for $\vec{v} \in V_p$ and $\vec{\phi} \in W_0^{1,s}(\Omega)$.

We collect our results on $\vec{S}$ and $\vec{T}$ to deduce properties
of $\vec{P}$: 
\begin{cor} \label{cor:t_estim}
	For $p$, $q$ and $s$ as in Lemma
	\ref{lem:t2_str_contin}, the operator $\vec{P}$ defined in
	\eqref{eq:defi_P} is well-defined, bounded and continuous on
	$V_p \to W_0^{1,s}(\Omega)\d$. It is well-defined, bounded,
	continuous and pseudomonotone on $V_q \to V_q\d$ and it fulfils the
	estimate
	\begin{equation*}
		\fp{\vec{P}(\vec{u})}{\vec{u}} \geq G_1 \norm{\vec{Du}}_p^p - G_2
		\norm{\vec{Du}}_p^2 - G_3 \norm{\vec{Du}}_p 
	\end{equation*}
	for $\vec{u} \in V_q$. If furthermore
	\begin{equation} \label{eq:smallness2}
	(2-p)^{2-p} (p-1)^{p-1} G_1 \geq G_2^{p-1}G_3^{2-p},
	\end{equation}
	then
	\begin{equation*}
		\fp{\vec{P}(\vec{u})}{\vec{u}} \geq 0
	\end{equation*}
	holds for all $\vec{u} \in V_q$ with $\norm{\vec{Du}}_p = R :=
	\left[\frac{G_3}{(2-p)\,G_1}\right]^\frac{1}{p-1}$. 
\end{cor}
\begin{proof}
	Well-definedness, boundedness and continuity follow similarly to the
	properties of $\vec{S}$ and $\vec{T}$ in Lemmas
	\ref{lem:properties_S} and \ref{lem:t2_str_contin}.
	
	The continuity of $\S$ and its monotonicity, which follows from
	\eqref{eq:pdelta_growth1}$_1$, yield that $\vec{S}\colon V_q \to V_q\d$ is
	pseudomonotone. Lemma \ref{lem:t2_str_contin} shows that
	$\vec{T}\colon V_q \to V_q\d$ is strongly continuous and thus
	pseudomonotone. Therefore, the sum
	$\vec{P} = \vec{S} + \vec{T} - \vec{f}$ is also pseudomonotone.
	
	By Lemmas \ref{lem:S_estimate} and \ref{lem:t2_str_contin} and by
	the definition of the constants $G_1$, $G_2$, $G_3$
	in~\eqref{eq:g123}, we have
	\begin{align} \label{eq:t_estim1}
		\fp{\vec{P}(\vec{u})}{\vec{u}}
		&\geq \fp{\vec{S}(\vec{u})}{\vec{u}} -
		\abs{\fp{\vec{T}(\vec{u})}{\vec{u}}}
		- \abs{\fp{\vec{f}}{\vec{u}}} \nonumber
		\\
		&\geq G_1 \norm{\vec{Du}}_p^p - G_2 \norm{\vec{Du}}_p^2 - G_3
		\norm{\vec{Du}}_p
	\end{align}
	for any $\vec{u} \in V_q$. 
	
	Now assume that \eqref{eq:smallness2} holds. It follows
	$\left[\frac{(p-1)\,G_1}{G_2}\right]^{p-1} \geq \left[
	\frac{G_3}{(2-p)G_1}\right]^{2-p}$, so we may define
	$R := \left[\frac{G_3}{(2-p)G_1}\right]^\frac{1}{p-1}$ and obtain
	$\frac{(p-1)\,G_1}{G_2} \geq R^{2-p}$. Together, we get
	\begin{equation} \label{eq:t_estim2}
	G_1 R^p = (p-1)G_1 R^p + (2-p)G_1 R^p \geq G_2 R^2 + G_3 R.
	\end{equation}
	We use $\norm{\vec{Du}}_p = R$, insert \eqref{eq:t_estim2} into
	\eqref{eq:t_estim1} and obtain the result. 
\end{proof}
\begin{rem}
	(i)  Note that the dependence of the constants $G_i$, $i=1,2,3$, on
	$\norm{\vec{g}}_{1,s}$ stems only from the estimate of the
	convective term.\\[-3mm]
	
	(ii)  In order to prove $G_1 R^p - G_2 R^2 - G_3 R \geq 0$, we have split
	the positive summand $G_1 R^p$ into two parts using the weights
	$p-1$ and $2-p$.  By considering the weights as a free parameter,
	one can show that this choice is optimal.
\end{rem}

Now we are ready to complete the proof of Theorem
\ref{thm:main_thm}. Since the proof requires an approximation process
only if $p \in {\big ( \frac{2d}{d+2}, \frac{3d}{d+2} \big ]} $, we
handle the two cases separately.

\begin{proof}[Proof of Theorem \ref{thm:main_thm} in the case $p \in
	\big (\frac{3d}{d+2}, 2\big )$]
	In this case we have  $s =
	p$. By Lemma \ref{lem:div_eq}, we  find a function
	$\vec{g} \in W^{1,p}(\Omega)$ that solves the corresponding
	inhomogeneous divergence equation \eqref{eq:g_problem}.
	
	We consider the corresponding operator $\vec{P}$ defined in
	\eqref{eq:defi_P} and prove existence of a function
	$\vec{u} \in V_p$ which satisfies $\vec{P}(\vec{u}) = \vec{0}$.  The
	space $V_p$ is reflexive and separable as a closed subspace of
	$W^{1,p}(\Omega)$. In Corollary \ref{cor:t_estim} (with
	$q = s = p$), we proved that $\vec{P}$ is well-defined, bounded,
	continuous and pseudomonotone on $V_p \to V_p\d$ and we concluded
	from assumption \eqref{eq:g_smallness} that there exists a positive
	number~$R$ such that $\vec{P}$ is locally coercive with radius $R$.
	So we may apply the main theorem on pseudomonotone operators,
	Theorem \ref{thm:brezis}, to the operator $\vec{P}$ on the space
	$V_p$ and obtain a weak solution $\vec{u} \in V_p$ of
	$\vec{P}(\vec{u}) = \vec{0}$.
	
	By a standard characterization of weak gradient fields (cf.~\cite{Sohr}),
	this is equivalent to the existence of a pressure $\pi$ such that
	$(\vec{u}+\vec{g}, \pi) \in W^{1,p}(\Omega) \times L^{p'}(\Omega)$
	solves the original system \eqref{eq:main_problem}.
\end{proof}

\begin{proof}[Proof of Theorem \ref{thm:main_thm} in the case
	$p \in {\big ( \frac{2d}{d+2}, \frac{3d}{d+2} \big ]} $]
	In this case we have  $s = \big (\frac{p\d}{2}\big
	)' $. By assumption \ref{assum:smallness}, there is a function
	$\vec{g} \in W^{1,s}(\Omega)$ which solves the inhomogeneous
	divergence equation \eqref{eq:g_problem} and satisfies
	\eqref{eq:g_smallness}. We prove the existence of a function
	$\vec{u} \in V_p$ that solves $\vec{P}(\vec{u}) = \vec{0}$ for the
	corresponding operator $\vec{P}$ from \eqref{eq:defi_P}.  For
	regularization, we choose some $q > s$ with $q > 2$ and consider the
	symmetric $q$-Laplacian $\vec{A} \colon V_q \to V_q\d$ defined via
	\begin{equation*}
		\fp{\vec{A}(\vec{u})}{\vec{\phi}} := \fp{\abs{ \vec{Du}}^{q-2}
			\vec{Du}}{\vec{D\phi}}. 
	\end{equation*}
	The operator $\vec{A}$ is well-defined, bounded, continuous and
	monotone.
	
	We work in the reflexive and separable spaces $V_q^n$, which are
	defined as the space $V_q$ equipped with the equivalent norms
	$\lVert {\vec{u}} \rVert _{q, n} := \max \left\{ n^\frac{-2}{2q-1}
	\norm{\vec{Du}}_q, \norm{\vec{Du}}_p \right\}$.  For sufficiently
	large $n$, we want to establish the existence of solutions $\vec{u^n} \in V_q^n$ to
	the equation
	\begin{equation} \label{eq:tq}
	\fp{\vec{P}(\vec{u^n})}{\vec{\phi}} + \tfrac{1}{n}
	\fp{\vec{A}(\vec{u^n})}{\vec{\phi}} = 0 
	\end{equation}
	for $\vec{\phi} \in V_q$, which shall approximate a solution of the
	original equation.  The operator $\vec{P}$ is pseudomonotone,
	continuous and bounded by Corollary~\ref{cor:t_estim} and the same
	holds for $\vec{A}$ and their sum $\vec{P}+\frac{1}{n}\vec{A}$.  We
	prove local coercivity of $\vec{P} + \frac{1}{n} \vec{A}$ with
	radius $R := \left[\frac{G_3}{(2-p)G_1}\right]^\frac{1}{p-1}$. So,
	let $\lVert {\vec{u}} \rVert _{q, n} = R$. If
	$n^\frac{-2}{2q-1} \norm{\vec{Du}}_q \leq \norm{\vec{Du}}_p$, we
	have $\norm{\vec{Du}}_p = \lVert {\vec{u}} \rVert _{q, n} = R$ and we get
	$\fp{\vec{P}(\vec{u})}{\vec{u}} \geq 0$ by assumption
	\eqref{eq:g_smallness} and Corollary \ref{cor:t_estim}.  Otherwise,
	suppose $n^\frac{-2}{2q-1} \norm{\vec{Du}}_q > \norm{\vec{Du}}_p$,
	so $R = n^\frac{-2}{2q-1} \norm{\vec{Du}}_q$ and
	$R > \norm{\vec{Du}}_p$. This, Corollary \ref{cor:t_estim} and the
	Sobolev embedding $V_q \injto V_p$ imply
	\begin{align*}
		\fp{\vec{P}(\vec{u})}{\vec{u}} + \tfrac{1}{n}\fp{\vec{A}(\vec{u})}{\vec{u}} 
		&\geq G_1 \norm{\vec{Du}}_p^p - G_2 \norm{\vec{Du}}_p^2 - G_3
		\norm{\vec{Du}}_p + \tfrac{1}{n}\norm{\vec{Du}}_q^q
		\\ 
		&> -G_2 R^{\,2} - G_3 R + n^\frac{1}{2q-1} R^{\,q}\!.
	\end{align*}
	As $n^\frac{1}{2q-1}$ grows to infinity, the latter expression
	becomes positive for any $R > 0$ and sufficiently large $n$.
	
	Thus, the existence Theorem \ref{thm:brezis} gives us solutions
	$\vec{u^n} \in V_q^n$ of \eqref{eq:tq} with
	\begin{equation} \label{eq:bound_un} \max \left\{ n^\frac{-2}{2q-1}
	\norm{ \vec{Du}}_q\!, \norm{\vec{Du}}_p \right\} =
	\lVert {\vec{u^n}} \rVert _{q, n} \leq R
	\end{equation}
	and the bound $R$ holds uniformly with respect to $n$.
	
	We switch to a weakly convergent (and renamed) subsequence
	$\vec{u^n}\weakto \vec{u}$ in $V_p$. The bound \eqref{eq:bound_un}
	implies  $\norm{n^{-1}\vec{A}(\vec{u^n})}_{q'} = n^{-1}
	\norm{\vec{Du^n}}_{q}^{q-1} \leq n^{\frac{-1}{2q-1}} R^{\,q-1} \to
	0$ as $n \to \infty $.
	
	We apply Theorem \ref{thm:ident_limits} with the shifted
	extra stress tensor $\A(\cdot) := \S(\cdot+\vec{Dg})$, 
	$\vec{B} := \vec{T} - \vec{f}\colon V_p \to V_s\d$,   $r := q$, $X := V_q$ and
	$\vec{A_n} := n^{-1}\vec{A}\colon V_q \to V_q\d\!$ and obtain that
	$\vec{u}$ solves $\vec{P}(\vec{u}) = \vec{0}$ weakly.
	
	Similarly to the proof in the first case, we obtain a pressure $\pi$
	such that the pair
	$(\vec{u}+\vec{g}, \pi) \in W^{1,p}(\Omega) \times L^{s'}(\Omega)$
	is a weak solution of \eqref{eq:main_problem}.
\end{proof}

\subsection{Less regular data} \label{sub:less_regularity}

In Theorem \ref{thm:main_thm}, we demanded additional regularity of
the data: we required $\vec{g} \in W^{1,s}(\Omega)$ with
$s = \big (\frac{p\d}{2}\big )' > p$ in the case
$p \in \big (\frac{2d}{d+2}, \frac{3d}{d+2}\big )$, while the solution
$\vec{v}$ is sought only in $W^{1,p}(\Omega)$. Thus, we want to
discuss whether this assumption is really necessary or if it may be
removed, perhaps for the price of more regular test functions. This
question only arises if
$p \in \big (\frac{2d}{d+2}, \frac{3d}{d+2}\big )$, since one has
$s = p$ and $\vec{g} \in W^{1,s}(\Omega) = W^{1,p}(\Omega)$ in the
other case.

In the proof of Theorem \ref{thm:main_thm}, we used the additional
regularity of our data to get more convenient estimates of the
convective term in Lemma \ref{lem:t2_str_contin}. Since these
estimates are mainly based on the H\"older inequality, one has to use
a stronger norm of $\vec{u}$ if only $\vec{g}\in W^{1,p}(\Omega)$ is
presumed. By \eqref{eq:t2_eq4} and the H\"older inequality, one
obtains the following result similar to Lemma \ref{lem:t2_str_contin}:

\begin{lem}[Properties of $\vec{T}$] \label{lem:t2_str_contin2}
	Let $p \in \big (\frac{2d}{d+2}, \frac{3d}{d+2}\big )$ and $q$ be so
	large that it holds both $q > s = \big (\frac{p\d}{2}\big )'$ and
	$\frac{1}{q'} \geq \frac{1}{p}+\frac{1}{p\d}$. For any given
	function $\vec{g} \in W^{1,p}(\Omega)$, the operator $\vec{T}$
	has the upper bound 
	\begin{align*}
		\abs{\fp{\vec{T}(\vec{u})}{\vec{u}}}
		&\leq c_\text{Sob} c_\text{Korn}^2 \left( \norm{\vec{Dg}}_p +
		\tfrac{1}{2} \norm{\div \vec{g}}_p \right) \norm{\vec{Du}}_p
		\norm{\vec{Du}}_q
		\\ 
		&\quad + c_\text{Sob} \left( \norm{\vec{g}}_{1,p}^2 +
		c_\text{Korn} \norm{\div \vec{g}}_p \norm{\vec{g}}_{1,p} \right)
		\norm{\vec{Du}}_q 
	\end{align*}
	for all $\vec{u} \in V_q$, where $C_\text{Sob}$ are Sobolev
	embedding constants.
\end{lem}

This and Lemma \ref{lem:S_estimate} yield an alternative lower bound
of $\vec{P}$: 

\begin{cor}[Alternative estimate of $\vec{P}$] \label{cor:t_estim2}
	Let $p \in \big (\frac{2d}{d+2}, \frac{3d}{d+2}\big )$ and $q$ be so
	large that it holds both $q > s = \big (\frac{p\d}{2}\big )'$ and
	$\frac{1}{q'} \geq \frac{1}{p}+\frac{1}{p\d}$. For any given
	function $\vec{g} \in W^{1,p}(\Omega) \setminus \{\vec{0}\}$, there
	are constants $F_1, F_2, G_1 \geq 0$ with $F_1, G_1 > 0$ such that
	it holds
	\begin{equation*}
		\fp{\vec{P}(\vec{u})}{\vec{u}} \geq G_1 \norm{\vec{Du}}_p^p -
		F_1 \norm{\vec{Du}}_q - F_2 \norm{\vec{Du}}_p
		\norm{\vec{Du}}_q 
	\end{equation*}
	for all $\vec{u} \in V_q$.
\end{cor}

To the authors' knowledge, a substantial improvement of the
estimates in Lemma~\ref{lem:t2_str_contin2} and Corollary
\ref{cor:t_estim2} is not available. So we ask whether the proof of
Theorem \ref{thm:main_thm} can be modified such that it works out with
\emph{these} estimates.

The main theorem on pseudomonotone operators, Theorem
\ref{thm:brezis}, which was used to obtain approximate solutions in
some smoother space $X = V_p \cap Y$ already gave a priori estimates
for these approximate solutions. These a priori bounds were needed to
establish a weak accumulation point.  The next Lemma shows that it is
impossible to obtain approximate solutions which are coming with an a
priori bound in $V_p$, by Brouwer's fixed point theorem/the main
theorem on pseudomonotone operators.

\begin{lem}[Limits for the applicability of pseudomonotone operator
	theory for solving~\eqref{eq:main_problem}] \label{lem:no_brouwer2}
	Let $F_1, G_1, R, 1<p<2<q$ be arbitrary, positive constants,
	$Y \subset L^1(\Omega)$ be a Banach space with norm $\norm{\cdot}_Y$
	and assume there is no continuous embedding $V_p \injto Y\!$. Define
	the operators $P_n\colon V_p \cap Y \to \R$ via
	\begin{equation*}
		P_n(\vec{u}) := G_1 \norm{\vec{Du}}_p^p + \tfrac{1}{n}
		\norm{\vec{u}}_Y^q - F_1 \norm{\vec{u}}_Y\!. 
	\end{equation*}
	Consider Banach spaces $Y_n := (V_p \cap Y, \norm{\cdot}_{Y_n})$
	which satisfy $Y_n\injto Y$ and
	${\norm{\vec{D}\cdot}_p \leq \norm{\cdot}_{Y_n}}$ for all $n \in
	\N$. Then, for all sufficiently large $n$ there exists a function
	$\vec{u^n} \in V_p \cap Y$ with $\norm{\vec{u^n}}_{Y_n}=R$ and
	\begin{equation} \label{eq:eq1}
	P_n(\vec{u^n}) < 0.
	\end{equation}
\end{lem}

\begin{rem}[Discussion of the assumptions in Lemma
	\ref{lem:no_brouwer2}] \label{rem:lem24_assumptions}
	The fact that the space $V_p \cap Y$, where we seek for approximate
	solutions, needs to be strictly smoother than $V_p$ is caught up in
	the assumption that $V_p$ does not embed continuously into $Y\!$.
	
	The term $\frac{1}{n} \norm{\vec{u}}_Y^q$ in the definition of~$P_n$
	may be equivalently replaced by any operator that is coercive on $Y$
	damped by some factor which is decreasing in $n$.
	
	The existence of embeddings $Y_n \injto Y$ is a natural assumption
	for the intersection of Banach spaces. The requirement
	$\norm{\vec{D} \cdot}_p \leq \norm{\cdot}_{Y_n}$ for all $n$ implies
	the existence of embeddings $Y_n \injto V_p$ and further means that
	a-priori estimates of the form $\norm{\vec{u^n}}_{Y_n}< c$ imply
	uniform boundedness in the weaker norm
	$\norm{\vec{D} \vec{u^n}}_p < c$ which is a necessary element in the
	proof of Theorem \ref{thm:main_thm}.
	
	Typically, one chooses $Y=V_q$ or $Y=L^q(\Omega)$ for some large
	number $q \in\R$ and works with the weighted sum norm
	$\norm{\cdot}_{Y_n} = \norm{\vec{D} \cdot}_p + \tfrac{1}{n}
	\norm{\cdot}_Y$. Obviously, the assumptions from Lemma
	\ref{lem:no_brouwer2} are fulfilled for such choices.
\end{rem}

\begin{proof}[Proof of Lemma \ref{lem:no_brouwer2}]
	Since $Y \subset L^1(\Omega)$, the intersection $V_p \cap Y$ is
	well-defined.  For each $n\in\N$, we define
	$f(n) := \inf_{u\in Y_n\setminus\{0\} }
	\frac{\norm{\vec{u}}_{Y_n}}{\norm{\vec{u}}_Y}$.  The embedding
	$Y_n\injto Y$ and the assumption on $\norm{\cdot}_{Y_n}$ imply that
	$f(n)$ is strictly positive and that
	\begin{equation} \label{eq:eq2}
	\max \left\{ \norm{\vec{D} \vec{u}}_p, f(n) \norm{\vec{u}}_Y \right\}
	\leq R
	\end{equation}
	holds for all $\vec{u} \in Y_n$ with $\norm{\vec{u}}_{Y_n}=R$. 
	
	We define the constant
	$c_s := \sup_{\vec{u}\in Y_n\setminus\{0\} } \frac{\norm{\vec{D}
			\vec{u}}_p}{\norm{\vec{u}}_Y} \in (0, \infty]$ and use the
	convention $\frac{t}{\infty}=0$ for any $t\in\R$.
	
	Our first step is the indirect proof of an upper bound on
	$f$. Therefore, we choose $c_1\in\R$ so large that
	$\frac{R^{\,q-1}}{c_1} - \frac{F_1}{2} < 0$ is satisfied.
	
	\emph{Step 1: We prove that \eqref{eq:eq1} is true for those
		$n \in \N$ with $f(n)^{q-1} \geq c_1 n^{-1}$.}
	
	By \eqref{eq:eq2}, we obtain
	\begin{equation*}
		\frac{1}{n} \norm{\vec{u}}_Y^q - \frac{F_1}{2} \norm{\vec{u}}_Y
		\leq \norm{\vec{u}}_Y \left[ \frac{R^{\,q-1}}{n f(n)^{q-1}} - \frac{F_1}{2} \right]
		\leq \norm{\vec{u}}_Y \left[ \frac{R^{\,q-1}}{c_1} - \frac{F_1}{2} \right]
		< 0
	\end{equation*}
	for any $\vec{u}\in Y_n$ with $\norm{\vec{u}}_{Y_n}=R$. Since there
	is no embedding $V_p \injto Y\!$, we may find some
	$\vec{u^n} \in Y_n$ with $\norm{\vec{u^n}}_{Y_n}=R$ and
	$G_1 \norm{\vec{D} \vec{u^n}}_p^p < \frac{F_1}{2}
	\norm{\vec{u^n}}_Y$. Together, we obtain
	$G_1 \norm{\vec{D} \vec{u^n}}_p^p + \frac{1}{n} \norm{\vec{u^n}}_Y^q
	- F_1 \norm{\vec{u^n}}_Y < 0$, which is \eqref{eq:eq1}.
	
	In the following, we only consider those $n \in \N$ with
	$f(n)^{q-1} < c_1 n^{-1}$.
	
	\emph{Step 2: Computation of suitable norms for functions
		$\vec{u^n}$ such that \eqref{eq:eq1} becomes true.}
	
	We choose $c_2 > 1$ so large that
	$\frac{c_2\, R^{\,q-1}}{F_1} \geq c_1$ and define the auxiliary
	functions $t_n\colon \R\to\R$ via
	\begin{equation*}
		t_n(x) := \frac{c_2}{n} x^{q-1} - F_1.
	\end{equation*}
	We claim that for all sufficiently large $n$, the equation
	$t_n(x)=0$ has a solution  $y_n \in \big ( \frac{R}{c_s}, \frac{R}{f(n)} \big )$.
	Since the upper bound on $f$ implies $f(n) \to 0$ as $n \to\infty$,
	it holds $\frac{R}{c_s} < \frac{R}{f(n)}$ for sufficiently large $n$
	and the interval $\big ( \frac{R}{c_s}, \frac{R}{f(n)} \big )$ is
	not empty.  We have
	$t_n\! \big (\frac{R}{c_s}\big ) = \frac{c_2 R^{\,q-1}}{n\,
		c_s^{q-1}} - F_1 < 0$ for sufficiently large $n$ and the
	definitions of $c_1$ and $c_2$ imply
	$\frac{c_2 R^{\,q-1}}{F_1} \geq c_1 > n f(n)^{q-1}$, thus
	$t_n\! \big (\frac{R}{f(n)}\big ) = \frac{c_2\, R^{\,q-1}}{n\,
		f(n)^{\,q-1}} - F_1 > 0$. The existence of zeroes $y_n$ then
	follows from the mean value theorem.
	
	Right from the definition of $y_n$, we obtain
	$\frac{1}{n}\, y_n^{\,q} - \frac{1}{c_2} F_1 y_n = 0$ and
	\begin{equation*}
		\left[1-\frac{1}{c_2}\right] F_1 y_n =
		\left[1-\frac{1}{c_2}\right] F_1 \left[ \frac{nF_1}{c_2}
		\right]^{\frac{1}{q-1}} > G_1 R^p 
	\end{equation*}
	for sufficiently large $n$.  Together, it follows
	\begin{equation*}
		G_1 R^p + \frac{1}{n} y_n^q - F_1 y_n   = G_1 R^p -
		\left[1-\frac{1}{c_2}\right] F_1 y_n + \frac{1}{n} y_n^q -
		\frac{1}{c_2} F_1 y_n    < 0.
	\end{equation*}
	Thus, any function $\vec{u^n} \in V_p \cap Y$ with
	$\norm{\vec{u^n}}_{Y_n}=R$ and $\norm{\vec{u^n}}_{Y}=y_n$ satisfies
	\begin{equation*}
		P_n(\vec{u^n}) \leq G_1 R^p + \frac{1}{n} y_n^q - F_1 y_n < 0.
	\end{equation*}
	
	\emph{Step 3: Construction of functions $\vec{u^n}\in V_p \cap Y$
		with prescribed norms.}
	
	It remains to prove that for any
	$y_n \in \big ( \frac{R}{c_s}, \frac{R}{f(n)} \big )$, there is a
	function $\vec{u^n} \in V_p \cap Y$ with $\norm{\vec{u^n}}_{Y_n}=R$
	and $\norm{\vec{u^n}}_{Y}=y_n$.
	
	Assume that $\norm{\vec{u}}_{Y}>y_n$ for all $\vec{u} \in Y_n$ with
	$\norm{\vec{u}}_{Y_n}=R$. Since
	$y_n \in \big( \frac{R}{c_s}, \frac{R}{f(n)} \big)$, it holds
	$c_s > \frac{R}{y_n}$. By the definition of $c_s$, this implies the
	existence of a function $\vec{u} \in Y_n$ such that
	$\norm{\vec{D} \vec{u}}_p > \frac{R}{y_n} \norm{\vec{u}}_Y$. Without
	loss of generality, we may scale $\vec{u}$ such that
	$\norm{\vec{u}}_{Y_n}=R$. We compile these estimates of $\vec{u}$
	and apply \eqref{eq:eq2} to obtain
	\begin{equation*}
		R < \frac{R}{y_n} \norm{\vec{u}}_Y 
		< \norm{\vec{D} \vec{u}}_p
		\leq R,
	\end{equation*}
	which is impossible.
	
	Now assume that $\norm{\vec{u}}_{Y}<y_n$ for all
	$\vec{u} \in Y_n$ with $\norm{\vec{u}}_{Y_n}=R$. As above, we obtain
	$\frac{R}{y_n} > f(n)$. This and the definition of $f(n)$ imply the
	existence of a function $\vec{u} \in Y_n$ such that it holds
	$\norm{\vec{u}}_{Y_n} < \frac{R}{y_n} \norm{\vec{u}}_Y$. We may
	scale $\vec{u}$ such that $\norm{\vec{u}}_{Y_n} = R$ and it follows
	\begin{equation*}
		R = \norm{\vec{u}}_{Y_n} < \frac{R}{y_n} \norm{\vec{u}}_Y < R, 
	\end{equation*}
	which is again a contradiction. 
	
	Hence, there are $\vec{v^n}, \vec{w^n} \in Y_n$ with
	$\norm{\vec{v^n}}_{Y_n} = \norm{\vec{w^n}}_{Y_n} = R$ and
	$\norm{\vec{v^n}}_{Y} \leq y_n \leq \norm{\vec{w^n}}_{Y}$ for every
	$n\in\N$. By assumption, the mapping
	$\vec{u}\mapsto\norm{\vec{u}}_Y$ is continuous on $Y_n$. We apply
	the mean value theorem on the (path-connected) sphere
	$\left\{ \norm{\vec{\cdot}}_{Y_n} = R \right\}$ and obtain functions
	$\vec{u^n} \in V_p \cap Y$ with $\norm{\vec{u^n}}_{Y_n}=R$ and
	$\norm{\vec{u^n}}_{Y} = y_n$. In Step 2 we proved that such an
	element $\vec{u^n}$ solves \eqref{eq:eq1} for all sufficiently large
	$n$.
\end{proof}

Lemma \ref{lem:no_brouwer2} shows that, without assuming additional
regularity, it is impossible to find a radius $R$ such that local
coercivity is fulfilled and Brouwer's fixed point theorem becomes
applicable. Consequently, the authors view the existence proof in
\cite[Theorem 1.3]{Sin} with suspicion; in particular, the requirements for
Brouwer's fixed point theorem do not seem to be satisfied in our eyes.
We conclude from Lemma \ref{lem:no_brouwer2} that it is impossible to
modify the proof of existence theorem \ref{thm:main_thm} such that it
avoids the critical regularity assumption within the framework of
pseudomonotone operator theory.

\bibliographystyle{abbrv}
\bibliography{literatur}

\begin{thebibliography}{10}

\bibitem{mikelic}
E.~Blavier and A.~Mikeli{\'c}.
\newblock {On the stationary quasi‐Newtonian flow obeying a power‐law}.
\newblock {\em M2AS}, 18:927--948, 1995.

\bibitem{bo1}
M.~Bogovskii.
\newblock Solution of the first boundary value problem for the equation of
  continuity of an incompressible medium.
\newblock {\em Dokl. Akad. Nauk SSSR}, 248:1037--1040, 1979.
\newblock English transl. in Soviet Math. Dokl. {\bf 20} (1979), 1094--1098.

\bibitem{bo2}
M.~Bogovskii.
\newblock Solution of some vector analysis problems connected with operators
  div and grad.
\newblock {\em Trudy Seminar S.L. Sobolev, Akademia Nauk SSSR}, 80:5--40, 1980.

\bibitem{BoyerFabrie}
F.~Boyer and P.~Fabrie.
\newblock {\em Mathematical tools for the study of the incompressible
  {N}avier-{S}tokes equations and related models}, volume 183 of {\em Applied
  Mathematical Sciences}.
\newblock Springer, New York, 2013.

\bibitem{dms}
L.~Diening, J.~M{\'a}lek, and M.~Steinhauer.
\newblock On {L}ipschitz truncations of {S}obolev functions (with variable
  exponent) and their selected applications.
\newblock {\em ESAIM: Control, Opt. Calc. Var.}, 14(2):211--232, 2008.

\bibitem{fms2}
J.~Frehse, J.~M{\'a}lek, and M.~Steinhauer.
\newblock On analysis of steady flows of fluids with shear-dependent viscosity
  based on the {L}ipschitz truncation method.
\newblock {\em SIAM J. Math. Anal.}, 34(5):1064--1083 (electronic), 2003.

\bibitem{Gagliardo}
E.~Gagliardo.
\newblock Caratterizzazioni delle tracce sulla frontiera relative ad alcune
  classi di funzioni in {$n$} variabili.
\newblock {\em Rend. Sem. Mat. Univ. Padova}, 27:284--305, 1957.

\bibitem{Gajewski}
H.~Gajewski, K.~Gr\"{o}ger, and K.~Zacharias.
\newblock {\em Nichtlineare {O}peratorgleichungen und
  {O}peratordifferentialgleichungen}.
\newblock Akademie-Verlag, Berlin, 1974.

\bibitem{Galdi}
G.~P. Galdi.
\newblock {\em An introduction to the mathematical theory of the
  {N}avier-{S}tokes equations. Steady-state problems}.
\newblock Springer Monographs in Mathematics. Springer, New York, 2011.

\bibitem{MA}
J.~Je{\ss{}}berger.
\newblock Existence of weak solutions for inhomogeneous generalized
  navier-stokes equations.
\newblock Master thesis, Albert-Ludwigs-Universit\"at Freiburg, January 31,
  2020.

\bibitem{lions-quel}
J.~Lions.
\newblock {\em Quelques M\'ethodes de R\'esolution des Probl\`emes aux Limites
  Non Lin\'eaires}.
\newblock Dunod, Paris, 1969.

\bibitem{Maly}
J.~Mal\'{y} and W.~P. Ziemer.
\newblock {\em Fine regularity of solutions of elliptic partial differential
  equations}, volume~51 of {\em Mathematical Surveys and Monographs}.
\newblock American Mathematical Society, Providence, RI, 1997.

\bibitem{r-mol-inhomo}
E.~Molitor and M.~R{\r u}{\v z}i{\v c}ka.
\newblock On inhomogeneous $p$--{N}avier--{S}tokes systems.
\newblock In V.~Radulescu, A.~Sequeira, and V.~Solonnikov, editors, {\em Recent
  Advances in PDEs and Applications}, volume 666 of {\em Contemp. Math.}, pages
  317--340. AMS Proceedings, 2016.

\bibitem{Cetraro}
M.~R{\r{u}}{\v{z}}i{\v{c}}ka.
\newblock Analysis of generalized {N}ewtonian fluids.
\newblock In {\em Topics in mathematical fluid mechanics}, volume 2073 of {\em
  Lecture Notes in Math.}, pages 199--238. Springer, Heidelberg, 2013.

\bibitem{Saramito}
P.~Saramito.
\newblock {\em Complex fluids}, volume~79 of {\em Math\'{e}matiques \&
  Applications (Berlin) [Mathematics \& Applications]}.
\newblock Springer, Cham, 2016.

\bibitem{Sin}
C.~Sin.
\newblock The existence of weak solutions for steady flows of
  electrorheological fluids with nonhomogeneous dirichlet boundary condition.
\newblock {\em Nonlinear Analysis}, 163:146--162, 2017.

\bibitem{Sohr}
H.~Sohr.
\newblock {\em The {N}avier-{S}tokes equations}.
\newblock Modern Birkh\"{a}user Classics. Birkh\"{a}user/Springer Basel AG,
  Basel, 2001.

\bibitem{Zeidler2B}
E.~Zeidler.
\newblock {\em Nonlinear functional analysis and its applications. {II}/{B}}.
\newblock Springer-Verlag, New York, 1990.
\newblock Nonlinear monotone operators.

\end{thebibliography}
\addcontentsline{toc}{section}{References}

\end{document}